\documentclass[11pt,a4paper]{article}
\usepackage[utf8]{inputenc}
\usepackage{amsmath}
\usepackage{amsthm}
\usepackage{amsfonts}
\usepackage{amssymb}
\usepackage{fancyvrb}
\usepackage{graphicx}
\usepackage{listings}
\usepackage{hyperref}
\usepackage{url}
\usepackage{algorithm}
\usepackage{algpseudocode}
\usepackage{pdflscape}
\usepackage{tikz}
\usepackage{longtable}
\usepackage{array}
\usepackage{booktabs}
\usetikzlibrary{matrix}
\usepackage[backend=biber,style=numeric]{biblatex}
\usepackage[left=2cm, right=2cm, top=2cm, bottom=2cm]{geometry}

\definecolor{keywordcolor}{rgb}{0.7, 0.1, 0.1}   
\definecolor{keywordcolor}{cmyk}{0,0.90,0.86,0} 
\definecolor{tacticcolor}{rgb}{0.0, 0.1, 0.6}    
\definecolor{tacticcolor}{cmyk}{1,0.1,0,0.1} 
\definecolor{commentcolor}{rgb}{0.4, 0.4, 0.4}   
\definecolor{commentcolor}{cmyk}{1,0.1,0,0.1} 
\definecolor{symbolcolor}{rgb}{0.0, 0.1, 0.6}    
\definecolor{symbolcolor}{cmyk}{1,0.1,0,0.1} 
\definecolor{sortcolor}{rgb}{0.1, 0.5, 0.1}      
\definecolor{sortcolor}{cmyk}{0.20,0,1,0.19}    
\definecolor{attributecolor}{rgb}{0.7, 0.1, 0.1} 
\definecolor{attributecolor}{cmyk}{0,0.90,0.86,0} 

\def\mmlean{\lstinline[mathescape]}

\newtheorem*{theorem}{Theorem}
\newtheorem*{lemma}{Lemma}
\newtheorem*{corollary}{Corollary}
\newtheorem*{definition}{Definition}

\title{A Lean formalisation of the disproof of Ramaekers's conjecture}
\author{Jeremy Tan}

\begin{document}

\maketitle

\begin{abstract}
We describe Redhill, a Lean formalisation of the major results in Hölzl, Kleine and Stephan's ``Improved Lower Bounds for Strong $n$-Conjectures'' \cite{strongN}. Design decisions regarding types and data structures are justified as enabling \textit{canonical} versions of the paper's claims, and from there the most direct proofs of said claims. The formalisation process also uncovered many errors, omissions and optimisation opportunities in the informal proof, all of which are addressed in the final library.

Although Redhill is meant to stand alone, a few lemmas were upstreamed to the primary Lean mathematics library \verb|mathlib|. We discuss the tension between upstreaming and leaving lemmas in separate libraries.

Unlike many other contemporary formalisation projects, Redhill was completed without any assistance from large language models (LLMs). We compare our work with LLM-assisted projects, in particular an extension of \cite{strongN} to $n=4$, aiming to establish design and style guidelines that define good Lean formalisations.
\end{abstract}

\section{Introduction} \label{introduction}

Part of the promise of the fourth version of the Lean programming language and theorem prover \cite{lean4}, with its emphasis on efficiency, extensibility in the language itself and ease of use, was to give every correct theorem in mathematics a readable machine-checkable proof. We say a formal proof is \textit{readable} when a human can understand it without any more difficulty than a corresponding informal version. Such proofs bring out the best of formal mathematics: not only do they expose implicit assumptions and abuses of notation in informal proofs -- what Bauer called ``invisible mathematics'' \cite{invisible} -- they also invite post-formalisation improvements, extraction of intermediate lemmas and development of further theory, whether by humans, LLMs or a combination of the two.

Large-scale formalisation projects using Lean and its primary mathematics library \verb|mathlib| \cite{mathlib} began not long after the latter could express the most common mathematical constructs -- real numbers, complex numbers, groups, rings, etc. Since they were started well before LLMs became commonplace and the effort required to formalise their background theory was not always estimable, they often took the form of Polymath-style collaborations: project maintainers drew up a blueprint, dividing the proof into a web of sorries (unproved Lean statements) that human contributors then solved separately. Examples of such collaborations are:
\begin{itemize}
    \item The Liquid Tensor Experiment \cite{liquid-tensor}, an early application of Lean to frontier mathematics research. Unlike the three projects below, formalising the theorem statements was a difficult task in itself because the definitions involved were novel.
    \item Fermat's Last Theorem for regular primes \cite{flt-regular}. Class field theory, the basis for modern proofs of this result, had not yet been formalised as a cohesive unit, so the proof proceeded through a series of more elementary lemmas numbered in Hilbert's \textit{Zahlbericht} \cite{zahlbericht}. The upgrading of Lean to its fourth version and the ensuing port of \verb|mathlib| also occurred during this project, which the maintainers had to wrangle with.
    \item The Equational Theories project \cite{equational}, which explored the implications between magmas defined by a single equation. External automated theorem provers like Vampire \cite{vampire} were used for some tasks (proving an implication or refuting one via counterexample), with their output translated to Lean.
    \item The Carleson project \cite{carleson} to prove a wide generalisation of Carleson's 1966 result on convergence of Fourier series \cite{carleson-original}, which the present author contributed to significantly. Many tasks in this project's blueprint revolved around estimation of integrals, for which the parallel development of measure theory and Lebesgue integration in \verb|mathlib| was of paramount importance. Shortly after the last sorry was filled, the constant factor involved in its statement was optimised to about the tenth \textit{root} of its original value ($2^{443a^3}$ to $2^{44a^3}$), illustrating the usefulness of a good formalisation beyond its own body of code.
\end{itemize}

All these projects were conducted in the open, turning theorem proving into an instance of open-source collaborative coding. They also depended upon \verb|mathlib| and \textit{expected} to improve the latter as a by-product, either by upstreaming missing lemmas useful in a wider context or by suggesting definition and implementation changes to speed up proof checking (cf. section 9 of \cite{flt-regular}). Their success suggested that Lean's future and the future of mathematics would remain human-centred.

By the time the Carleson project finished in July 2025, however, LLMs had become capable of \textit{autoformalisation}: taking an informal proof as input (or generating one autonomously) and expressing it in Lean to validate the informal reasoning. Both dedicated models (Aristotle \cite{aristotle}, Seed Prover \cite{seedprover}) and adapted versions of general-purpose models (GPT, Claude) have been employed for this purpose; the convenience they provide -- a ``fully correct'' proof from a few prompts -- has enabled many individuals and small teams to formalise theorems in much less time than they would need to understand their informal proofs. Examples include:
\begin{itemize}
    \item The classification of convex regular polytopes in all dimensions \cite{platonic-classification}. Literature on this classical result is sparse, partly because of the influence of Coxeter's \textit{Regular Polytopes} \cite{coxeter} which relies on geometric intuition in many places, yet the formalisation took only five days. This project was submitted to the autoformalisation benchmark \verb|lean-eval| \cite{lean-eval}, which provided the specific formal statement to be proved.
    \item The strong perfect graph theorem \cite{pegden-spgt}, formalised with the help of a deterministic workflow ensuring that LLM agents make substantive progress on the target statements \cite{pegden-trellis}.
    \item The Feit--Thompson theorem and related foundational theorems in the classification of finite simple groups \cite{formatheoria}. The authors developed a more elaborate workflow capable of aligning and extracting common definitions and lemmas -- their goal was to generate a library that could be built upon like \verb|mathlib|. We will see in section \ref{llmcomp} that such agent coordination workflows are \textit{inherently inadequate} for the aforementioned goal.
    \item Anthropic's formalisation of the full Fermat's Last Theorem \cite{anthropic-flt} in 11 days, using orders of magnitude more computational power than the above three projects. Its approach is essentially that of Wiles's original proof from 1995, not the modern route developed in Imperial College London's \href{https://github.com/ImperialCollegeLondon/FLT}{FLT repository} \cite{buzzard}.
\end{itemize}

All these LLM-assisted projects have one thing in common:
\begin{center}
    \textit{They only care that their target theorems are true -- and sacrifice everything else for it.}
\end{center}

The general availability of autoformalisers has sparked a gold rush where the ores are the entire mathematics literature, the gold is formal proofs (regardless of quality) and the miners are impatient people using LLMs. \verb|lean-eval|'s setup, displaying an LLM leaderboard while allowing submissions to hide their code, has only accelerated this rush. There are projects such as \cite{kourovka} using LLMs with all due respect to humans, but they are a rarity.

Without the careful human guidance they deserve, the ``fully correct'' proofs coming out of LLMs represent the \textit{worst} kind of formal mathematics. Their definitions are non-canonical and their code is disorganised, and as a result they become write-only. Against the proliferation of ``merely true'' theorems and as a demonstration of high-quality Lean code, in this paper we present Redhill, a library formalising the disproof of Ramaekers's conjecture for all $n\ge4$ and other related theorems. The library may be found at \url{https://github.com/Parcly-Taxel/Redhill}.

The rest of this paper is organised as follows. Section \ref{leanstruct} gives an overview of what goes into a Lean proof from a theorem-proving perspective, in particular the tactics used to manipulate terms and goals. Section \ref{common-structures} discusses common constructs used throughout Redhill and surrounding design decisions, while section \ref{disproof} describes the major theorems formalised. Section \ref{upstreaming} is about upstreaming lemmas from dependent projects like Redhill to \verb|mathlib|, section \ref{llmcomp} compares our project to LLM-assisted projects (in particular the extension \cite{adamczewski} of \cite{strongN} to $n=4$) and section \ref{conclusion} concludes.

\section{Anatomy of a Lean proof} \label{leanstruct}

The Curry--Howard correspondence matches mathematical expressions (\textit{terms} in Lean's terminology) to types in a given type system. Lean's type system elaborates the calculus of inductive constructions \cite{lean4}, so a proposition \verb|p| is a term of type \verb|Prop| and \verb|p| is true or false as there exists or does not exist a term of type \verb|p|. Proving \verb|p| requires explicitly constructing a term of type \verb|p|, while disproving \verb|p| requires constructing a term of type \verb|False| from \verb|p| -- \href{https://lean-lang.org/doc/reference/latest/Basic-Propositions/Truth/#False}{by definition}, there are no terms of type \verb|False|.

A Lean\footnote{From here on any mention of Lean will implicitly include \Verb|mathlib|, as virtually every public Lean project proving a mathematical theorem has \Verb|mathlib| as a dependency. Some keywords and tactics are only defined in \Verb|mathlib|.} proof is a directed acyclic graph of terms culminating in a term of the desired type. Each term manifests as a code block or \textit{declaration} beginning with a keyword followed by its type. There are three main kinds of terms:
\begin{itemize}
    \item \textbf{Theorems} proving propositions are indicated by the keywords \verb|theorem| and \verb|lemma|. There is no semantic difference between these keywords; the distinction exists solely for humans to distinguish ``more important'' and ``less important'' theorems.
    \item \textbf{Definitions} constructing non-propositional data such as functions are indicated by the keywords \verb|def|, \verb|abbrev|, \verb|structure| and \verb|inductive|. A \verb|def| needs to be unfolded in a proof by tactics (described below) to access its definition, while an \verb|abbrev| does not. \verb|inductive| is the most general way of defining new \textit{types} in Lean; \verb|structure| defines a non-recursive type, akin to \verb|struct|s in C and C++.
    \item \textbf{Type classes} are indicated by the keyword \verb|class| and instantiated by \verb|instance|. Redhill does not define any type classes or instances because none of its results need to be \textit{ad hoc} polymorphic, but larger projects like \verb|mathlib| use them extensively \cite{instance-abuse}.
\end{itemize}

The following \verb|mathlib| \href{https://github.com/leanprover-community/mathlib4/blob/7364908aff0a24d686c6a59938e32a7c540ad50b/Counterexamples/HeawoodUnitDistance.lean#L38-L49}{snippet} illustrates \verb|def|, \verb|instance| and \verb|lemma| declarations in practice.

\begin{lstlisting}
def heawoodGraph : SimpleGraph (Fin 14) where
  Adj i j := (i - j).1 = 1            ∨ (j - i).1 = 1 ∨
             (i - j).1 = 5 ∧ Even j.1 ∨ (j - i).1 = 5 ∧ Even i.1

instance : DecidableRel heawoodGraph.Adj :=
  inferInstanceAs <| DecidableRel fun (i j : Fin 14) ↦
    (i - j).1 = 1 ∨ (j - i).1 = 1 ∨ (i - j).1 = 5 ∧ Even j.1 ∨ (j - i).1 = 5 ∧ Even i.1

lemma heawoodGraph_edgeFinset :
    heawoodGraph.edgeFinset = (univ.image fun i ↦ s(i, i + 1)) ∪
      (univ.filter (Even ·.1)).image fun i ↦ s(i, i + 5) := by
  decide +kernel
\end{lstlisting}

Of course, since multiple theorems may share the same proof foundations and the dependencies between lemmas are often sparse, declarations can be split across multiple files. \textit{How} to split to best reflect proof structure is something LLMs struggle at, as we will discuss in section \ref{llmcomp}.

\subsection{Tactic vocabulary} \label{tactic-vocab}

While it is theoretically possible to write a proof term for any true proposition, such terms are lengthy and unwieldy for all but the simplest propositions. \href{https://lean-lang.org/theorem_proving_in_lean4/Tactics}{\textit{Tactics}} get around this by representing a proof as a sequence of commands operating on the proposition to be proved, referred to as the \textit{goal}. A tactic can generate new goals, but every goal must eventually be closed.

Each tactic can be rephrased as one or more steps of an informal proof. Some of the tactics used in Redhill, together with their informal rephrasings, are described below.
\begin{itemize}
    \item \verb|exact t|, with \verb|t| a term, closes a goal whose type exactly matches \verb|t|'s. ``The goal is true by \verb|t|.''
    \item \verb|refine t| closes a goal whose type matches \verb|t|'s \textit{conditional} on any holes \verb|?_| appearing in \verb|t|. These holes become new goals to be closed. ``If the indicated holes are true, the goal is true by \verb|t|.''
    \item \verb|apply t| tries to close the goal by matching \verb|t| to it. Like \verb|refine|, the original goal is replaced by one new goal for every point where the matching did not succeed; it is more parsimonious to use \verb|exact| when no new goals are generated. ``If all differences between \verb|t| and the goal can be reconciled, the goal is true by \verb|t|.''
    \item \verb|rw [t1, t2, ...]|, where the arguments have types $a=b$ or $a\leftrightarrow b$, considers each argument in turn and replaces occurrences of $a$ in the goal by $b$. Prefixing an argument by $\leftarrow$ replaces occurrences of $b$ in the goal by $a$ instead. If the resulting goal is true by reflexivity, it is closed by implicitly running the \verb|rfl| tactic. ``The goal can be rewritten using \verb|t1, t2, ...|''
    \item \verb|simp only [t1, t2, ...]| iteratively rewrites using the provided arguments as many times as possible, closing the goal by \verb|rfl| if possible. It is more powerful but less precise than \verb|rw|: the arguments can either be of the types accepted by \verb|rw| or definitions to unfold, and \verb|simp only| can rewrite in places where \verb|rw| cannot, but there is no control over the extent of rewriting. ``The goal can be simplified using \verb|t1, t2, ...|''
    \item \verb|simp [t1, t2, ...]| is like \verb|simp only|, but adds a global set of lemmas tagged with the \verb|@[simp]| attribute to the arguments. Because this global set may change as a library develops, the Lean language reference \href{https://lean-lang.org/doc/reference/4.33.0/The-Simplifier/Terminal-vs-Non-Terminal-Positions}{recommends} that it only be used where it closes the goal. ``Simplifying gives the result (details are left to the reader).''
    \item \verb|simp_rw [t1, t2, ...]| is roughly equivalent to \verb|simp only| with each argument in order.
    \item \verb|cases t using pr with| performs case analysis on \verb|t|, generating one new goal for each case. \verb|using pr| is optional and specifies a different splitting method from the default one naturally arising from the definition of \verb|t|'s type. ``We split into different cases depending on \verb|t|.''
    \item \verb|induction t using pr with| extends \verb|cases| by making induction hypotheses available when proving case-specific goals. ``We proceed by induction on \verb|t|, using the induction principle \verb|pr|.''
    \item \verb|have : t| declares an arbitrary proposition, which must be proved as a new goal. ``We first prove the auxiliary lemma \verb|t|.''
    \item \verb|suffices t| produces the same goals as \verb|have t| but reverses their order, so the proof of the original goal from \verb|t| comes first. ``It suffices to prove \verb|t| because\textellipsis''
    \item \verb|decide| closes a goal with no free variables that can be effectively computed (a \verb|Decidable| instance for the goal exists or can be synthesised). Natural numbers, integers, finite types and finite sets are effectively computable for \verb|decide|'s purposes, but real numbers are not. ``The goal is true by computation in the Lean kernel.''
    \item \href{https://lean-lang.org/doc/reference/4.33.0/The--grind--tactic}{\Verb|grind [t1, t2, ...]|} closes a goal (and fails otherwise) by assuming its negation, then repeatedly deriving relations from lemmas tagged with the \verb|@[grind]| attribute, available hypotheses and provided arguments until reaching a contradiction. ``The goal is true by applying rules akin to a satisfiability modulo theories (SMT) solver. Details are left to the reader.''
    \item \href{https://lean-lang.org/doc/reference/4.33.0/The--grind--tactic/Linear-Integer-Arithmetic}{\Verb|lia [t1, t2, ...]|} is a tactic internally called by \verb|grind|, closing goals in linear integer arithmetic (linear constraints on integers and natural numbers with divisibility).
    \item \verb|ring| closes equations that follow from the axioms of commutative rings like $\mathbb Z$ or commutative semirings like $\mathbb N$. ``The goal is true by algebraic rearrangement.''
\end{itemize}

\section{Common structures} \label{common-structures}

The $n$-conjectures in the title of Redhill's primary source \cite{strongN} revolve around the \textit{quality} of particular sets of $n$-tuples of integers. The quality of one tuple $(a_1,\dots,a_n)$ is
\begin{equation}\label{eq:quality}
    \frac{\log(\max_i|a_i|)}{\log(\operatorname{rad}(\prod_i a_i))}
\end{equation}
where $\operatorname{rad}(n)$ is the radical, the product of all distinct prime factors of $n$. As the foundational definition, it was essential to formalise quality in Redhill in a canonical way: a human should have no doubt that the informal definition was translated correctly \textit{and} the formal definition should not pose unnecessary difficulties in declarations building on it. We chose the following \href{https://github.com/Parcly-Taxel/Redhill/blob/arxiv-fixed/Redhill/Common/Quality.lean#L24-L28}{formalisation} of (\ref{eq:quality}):
\begin{lstlisting}
open UniqueFactorizationMonoid in
/-- The quality of a single tuple.
This depends on Lean defining `log -x = log x` for all real `x`. -/
@[expose] noncomputable def tupleQuality (a : Fin n → ℤ) : ℝ≥0∞ :=
  .ofReal (log (maxAbs a) / log (radical (∏ i, a i) : ℤ))
\end{lstlisting}

The natural and hence canonical type for an $n$-tuple of integers is \mmlean{Fin n$\:\to\mathbb Z$}, a function from the indexing type \verb|Fin n| to the entry type $\mathbb Z$. The \href{https://github.com/leanprover/lean4/blob/470d5ce1400764999581fd26d5d72b00d990b0f4/src/Init/Prelude.lean#L2443-L2462}{documentation} for \verb|Fin n|, the natural numbers less than $n$, describes it as ``the canonical type with $n$ elements''.

\verb|maxAbs| is \href{https://github.com/Parcly-Taxel/Redhill/blob/arxiv-fixed/Redhill/Common/MaxAbs.lean#L21-L23}{defined} in an earlier file in Redhill:
\begin{lstlisting}
/-- The maximum absolute value of a tuple of integers (0 if empty). -/
def maxAbs (a : Fin n → ℤ) : ℕ :=
  univ.sup fun i ↦ (a i).natAbs
\end{lstlisting}

Here \verb|univ| is \href{https://github.com/leanprover-community/mathlib4/blob/331d5244f0d3aad530d9ab00ded135b4c7691502/Mathlib/Data/Fintype/Defs.lean#L92}{\Verb|Finset.univ|}, the finite set or \textit{finset} of all possible indices, and the appended \verb|.sup| is \href{https://lean-lang.org/doc/reference/latest/Terms/Function-Application/#generalized-field-notation}{generalised field notation} for \href{https://github.com/leanprover-community/mathlib4/blob/331d5244f0d3aad530d9ab00ded135b4c7691502/Mathlib/Data/Finset/Lattice/Fold.lean#L43}{\Verb|Finset.sup|}, which applies its argument to each element of the finset and returns the maximum value. We chose this definition over alternatives using \href{https://github.com/leanprover-community/mathlib4/blob/331d5244f0d3aad530d9ab00ded135b4c7691502/Mathlib/Data/Finset/Max.lean#L37}{\Verb|Finset.max|}/\href{https://github.com/leanprover-community/mathlib4/blob/331d5244f0d3aad530d9ab00ded135b4c7691502/Mathlib/Data/Finset/Max.lean#L134}{\Verb|Finset.max'|} because it requires no special-casing to evaluate to $0$ for the empty tuple.

The seemingly extraneous \href{https://github.com/leanprover-community/mathlib4/blob/4beb549110aa44b87d699166cece6e3642bddf34/Mathlib/Basic/ENNReal/Basic.lean#L233}{\Verb|ENNReal.ofReal|} in \verb|tupleQuality|'s definition, which casts the quotient to an extended non-negative real number (\href{https://github.com/leanprover-community/mathlib4/blob/4beb549110aa44b87d699166cece6e3642bddf34/Mathlib/Basic/ENNReal/Basic.lean#L104}{\Verb|ENNReal|}, \mmlean{ℝ≥0∞}), helps in two ways. Firstly, it helps the succeeding definition of the quality of a set of tuples, \href{https://github.com/Parcly-Taxel/Redhill/blob/arxiv-fixed/Redhill/Common/Quality.lean#L30-L33}{formalised} as follows:
\begin{lstlisting}
/-- The quality of a set of tuples, defined as the infimum of those numbers where
only finitely many tuples in the set have a strictly higher quality. -/
noncomputable def quality (A : Set (Fin n → ℤ)) : ℝ≥0∞ :=
  sInf {q | {a ∈ A | q < tupleQuality a}.Finite}
\end{lstlisting}

This is easily seen to be equivalent to enumerating the countably many tuples in \verb|A| and taking the limit superior of their qualities, which is how \cite{strongN} defines set quality. By avoiding filters and limits, proving theorems about \verb|quality| becomes easier, while using the \textit{complete} type \verb|ENNReal| obviates the need to prove boundedness conditions when applying theorems about the set infimum \href{https://github.com/leanprover-community/mathlib4/blob/4952ca049727ba250943fb27f5fb853a79d7ceda/Mathlib/Order/Bounds/Defs.lean#L77}{\Verb|sInf|}.

Secondly, the \textit{non-negativity} of \verb|ENNReal| reflects how the quotient passed to \verb|ENNReal.ofReal| is always non-negative, given Lean's definitions of division and the (natural) logarithm. They are total functions with definite values outside their traditionally understood domains: \verb|x / 0 = 0|, \verb|log 0 = 0| and \verb|log (-x) = log x|. The \verb|log| of an integer is thus non-negative, and we need not worry about \verb|ENNReal.ofReal| losing information by sending negative numbers to $0$.

The totality of division and \verb|log| also allows us to not impose well-definedness conditions on the argument to \verb|tupleQuality| for simplicity -- equation (\ref{eq:quality}) is well-defined iff $|\prod_i a_i|>1$, which every tuple in the disproof of Ramaekers's conjecture satisfies.

\subsection{Tuple predicates}

Each $n$-conjecture asserts something about the quality of the set of all $n$-tuples of integers summing to $0$ and satisfying some other conditions:
\begin{itemize}
    \item Browkin and Brzeziński's \href{https://github.com/Parcly-Taxel/Redhill/blob/arxiv-fixed/Redhill/Common/Conjectures.lean#L24-L31}{$n$-conjecture} \cite{bb94} requires $\gcd(a_1,\dots,a_n)=1$ and the \textit{subsum condition} that no proper subset of the $a_i$ sums to $0$. It asserts that the resulting set's quality equals $2n-5$, thereby reducing to one form of the $abc$-conjecture at $n=3$.
\begin{lstlisting}
/-- The subsum condition: any subset of the tuple summing to 0
is either empty or the whole tuple. -/
def SSC : Prop :=
  ∀ b, b.Nonempty → bᶜ.Nonempty → ∑ i ∈ b, a i ≠ 0

/-- The tuples in Browkin and Brzezinski's `n`-conjecture. `A(n)` in the paper. -/
def nConjectureTuples (n : ℕ) : Set (Fin n → ℤ) :=
  {a | ∑ i, a i = 0 ∧ SSC a ∧ univ.gcd a = 1}

/-- Browkin and Brzezinski's **`n`-conjecture** for a fixed `n`.
The conjecture itself is `∀ n ≥ 3, NConjecture n`. -/
def NConjecture (n : ℕ) : Prop :=
  quality (nConjectureTuples n) = (2 * n - 5 : ℕ)
\end{lstlisting}
    \verb|b| in Redhill's \href{https://github.com/Parcly-Taxel/Redhill/blob/arxiv-fixed/Redhill/Common/SubsumCondition.lean#L24-L27}{formalisation} of the subsum condition is the finset of indices defining the ``proper subset'' that must not sum to $0$.
    \item Browkin's \href{https://github.com/Parcly-Taxel/Redhill/blob/arxiv-fixed/Redhill/Common/Conjectures.lean#L33-L40}{strong $n$-conjecture} \cite{brow00} requires the $a_i$ to be pairwise coprime but imposes no subsum condition, and asserts that the resulting set's quality is finite.
\begin{lstlisting}
/-- A predicate stating that the given tuple's numbers are pairwise coprime. -/
abbrev PairwiseCoprime {n : ℕ} (a : Fin n → ℤ) : Prop :=
  Pairwise fun i j ↦ IsCoprime (a i) (a j)

/-- The tuples in Browkin's strong `n`-conjecture. `B(n)` in the paper. -/
def strongNConjectureTuples (n : ℕ) : Set (Fin n → ℤ) :=
  {a | ∑ i, a i = 0 ∧ PairwiseCoprime a}

/-- Browkin's **strong `n`-conjecture** for a fixed `n`.
The conjecture itself is `∀ n ≥ 3, StrongNConjecture n`. -/
def StrongNConjecture (n : ℕ) : Prop :=
  quality (strongNConjectureTuples n) < ⊤
\end{lstlisting}
    \item \href{https://github.com/Parcly-Taxel/Redhill/blob/arxiv-fixed/Redhill/Common/Conjectures.lean#L42-L49}{Ramaekers's conjecture} \cite{ramaekers} requires the $a_i$ to be pairwise coprime \textit{and} the subsum condition, and asserts that the resulting set's quality is $1$.
\begin{lstlisting}
/-- The tuples in Ramaekers's conjecture. `R(n)` in the paper. -/
def ramaekersTuples (n : ℕ) : Set (Fin n → ℤ) :=
  {a | ∑ i, a i = 0 ∧ SSC a ∧ PairwiseCoprime a}

/-- **Ramaekers's conjecture** for a fixed `n`.
The conjecture itself is `∀ n ≥ 3, RamaekersConjecture n`. -/
def RamaekersConjecture (n : ℕ) : Prop :=
  quality (ramaekersTuples n) = 1
\end{lstlisting}
\end{itemize}

To tackle the strong $n$-conjecture and Ramaekers's conjecture, the authors of \cite{strongN} fix a finset of natural numbers not less than $3$ as \textit{forbidden factors} for the $a_i$ and strengthen the subsum condition:
\begin{definition}
    A tuple $(a_1,\dots,a_n)$ satisfies the strong subsum condition if for any tuple $(b_1,\dots,b_n)$ with $b_i\in\{-1,0,1\}$ and $\sum_i b_ia_i=0$, $b_1=\dots=b_k$. This implies the subsum condition, where $b_i\in\{0,1\}$.
\end{definition}

In the sequel we abbreviate ``strong subsum condition'' to its Redhill name \href{https://github.com/Parcly-Taxel/Redhill/blob/arxiv-fixed/Redhill/Common/SubsumCondition.lean#L34-L36}{\Verb|StrongSSC|}.

\newpage

\begin{lstlisting}
/-- A subsum block for `a` is an index set that must be constant in any sign weighting
of the tuple's elements that leads to a zero sum. -/
def IsSubsumBlock (s : Finset (Fin n)) : Prop :=
  ∀ b : Fin n → SignType, ∑ i, b i * a i = 0 → ∃ c, ∀ i ∈ s, b i = c

/-- The strong subsum condition, defined as all `Fin n` being a subsum block. -/
def StrongSSC : Prop :=
  IsSubsumBlock a univ

/-- `U(F,n)` in the paper. -/
def factorFreeTuples (F : Finset ℕ) (n : ℕ) : Set (Fin n → ℤ) :=
  {a | ∑ i, a i = 0 ∧ StrongSSC a ∧ PairwiseCoprime a ∧ ∀ f ∈ F, ∀ i, ¬↑f ∣ a i}
\end{lstlisting}

We introduce the idea of a \href{https://github.com/Parcly-Taxel/Redhill/blob/arxiv-fixed/Redhill/Common/SubsumCondition.lean#L29-L32}{\textit{subsum block}} to help prove \verb|StrongSSC|.
\begin{definition}
    Given a tuple $(a_1,\dots,a_n)$, a sub-tuple $(a_{i_1},\dots,a_{i_k})$ is a subsum block if for any tuple $(b_1,\dots,b_n)$ with $b_i\in\{-1,0,1\}$ and $\sum_i b_ia_i=0$, $b_{i_1}=\dots=b_{i_k}$.
\end{definition}

It is clear that if $(a_{i_1},\dots,a_{i_k})=(a_1,\dots,a_n)$, \verb|StrongSSC| is satisfied. The following lemmas allow proving \verb|StrongSSC| by recursing on shorter tuples.
\begin{lemma}[\href{https://github.com/Parcly-Taxel/Redhill/blob/arxiv-fixed/Redhill/Common/SubsumCondition.lean\#L229}{\Verb|strongSSC\_tupReduce|}]
    If the tuple obtained by summing the elements of a subsum block into one new element satisfies \verb|StrongSSC|, so does the original tuple.
\end{lemma}
\begin{proof}
    Expand the definition of \verb|StrongSSC| for the original and reduced tuples.
\end{proof}
\begin{lemma}[\href{https://github.com/Parcly-Taxel/Redhill/blob/arxiv-fixed/Redhill/Common/SubsumCondition.lean\#L181}{\Verb|of\_sum\_natAbs\_lt|}]
    Given a tuple $(a_1,\dots,a_n,c_1,\dots,c_m)$, $(c_1,\dots,c_m)$ is a subsum block if for all non-constant tuples $(b_1,\dots,b_m)$ with $b_j\in\{-1,0,1\}$ we have
    \begin{equation*}
        \sum_{i=1}^n|a_i|<\left|\sum_{j=1}^m b_jc_j\right|.
    \end{equation*}
\end{lemma}
\begin{proof}
    Consider a sign-weighted sum of the entire tuple where $a_i$ has weight $d_i\in\{-1,0,1\}$:
    \begin{equation*}
        \sum_{i=1}^n d_ia_i+\sum_{j=1}^m b_jc_j
    \end{equation*}
    By the triangle inequality the first sum's magnitude cannot exceed $\sum_{i=1}^n|a_i|$, so the second sum cannot exactly cancel out the first sum \textit{unless $(b_1,\dots,b_m)$ is a constant tuple}, which is the definition of $(c_1,\dots,c_m)$ being a subsum block.
\end{proof}
\begin{corollary}[\href{https://github.com/Parcly-Taxel/Redhill/blob/arxiv-fixed/Redhill/Common/SubsumCondition.lean\#L199}{\Verb|pair\_of\_sum\_natAbs\_lt|}]
    Given a tuple $(a_1,\dots,a_n,x,y)$, $(x,y)$ is a subsum block if the following inequalities hold:
    \begin{equation*}
        \sum_{i=1}^n|a_i|<|x|\qquad\sum_{i=1}^n|a_i|<|y|\qquad xy\le0
    \end{equation*}    
\end{corollary}
\begin{proof}
    There are six possible ways to associate different signs to $x$ and $y$:
    \begin{equation*}
        (0,1)\qquad(0,-1)\qquad(1,0)\qquad(-1,0)\qquad(1,-1)\qquad(-1,1)
    \end{equation*}
    It is easy to check that the strict inequality condition in \verb|of_sum_natAbs_lt| is satisfied in each case. The last two cases satisfy the inequality because $xy\le0$ implies $|x-y|=|y-x|=|x|+|y|$.
\end{proof}

\subsection{Prime chains} \label{primechain}

To produce tuples in \verb|factorFreeTuples F n| we can start with a strictly increasing sequence of primes, which are obviously pairwise coprime. This is implemented in Redhill as \href{https://github.com/Parcly-Taxel/Redhill/blob/arxiv-fixed/Redhill/Common/PrimeChain.lean#L21-L25}{\Verb|primeChain s|}, where the first prime must be strictly greater than \verb|s|:
\begin{lstlisting}
/-- `primeChain s` is the lexicographically earliest sequence of primes
with `primeChain s 0 > s` and each succeeding element more than twice the last. -/
def primeChain (s : ℕ) : ℕ → ℕ
  | 0 => Nat.find (exists_infinite_primes (s + 1))
  | n + 1 => Nat.find (exists_infinite_primes (2 * primeChain s n + 1))
\end{lstlisting}

We actually require each prime to be strictly greater than \textit{twice} the previous prime, choosing the smallest admissible prime at each step using \href{https://github.com/leanprover-community/mathlib4/blob/9e6b3aac99b624d10c84653ab9c5357283b9b3b8/Mathlib/Data/Nat/Find.lean#L67}{\Verb|Nat.find|} to keep \verb|primeChain| concrete and computable. We use this bound to prove that \href{https://github.com/Parcly-Taxel/Redhill/blob/arxiv-fixed/Redhill/Common/PrimeChain.lean#L44-L46}{a modified tuple} summing to $0$ satisfies \verb|StrongSSC|:
\begin{lstlisting}
/-- An `(n + 2)`-tuple that satisfies the strong subsum condition if `0 < m ≤ s`. -/
@[expose] def chainTup (n m s : ℕ) (i : Fin (n + 2)) : ℤ :=
  i.addCases (primeChain s ·.1) fun | 0 => m | 1 => -(m + ∑ i ∈ range n, primeChain s i)
\end{lstlisting}

In this definition \href{https://github.com/leanprover/lean4/blob/c29b6dda4f7c20e3eeaa717c4e565663c5cfa364/src/Init/Data/Fin/Lemmas.lean#L1066}{\Verb|Fin.addCases|} concatenates the first $n$ primes in the chain $(p_1,\dots,p_n)$ and the pair $(m,-(m+\sum_{i=1}^n p_i))$ to produce a tuple of length $n+2$.

\begin{lemma}[\href{https://github.com/Parcly-Taxel/Redhill/blob/arxiv-fixed/Redhill/Common/PrimeChain.lean\#L138}{\Verb|strongSSC\_chainTup|}]
    Let $m$ and $s$ be natural numbers with $0<m\le s$ and $p_1,\dots,p_n$ be primes with $s<p_1$ and $2p_{i-1}<p_i$ for $i=2,\dots,n$. Then \verb|chainTup| $(p_1,\dots,p_n,m,-(m+\sum_{i=1}^n p_i))$ satisfies \verb|StrongSSC|.
\end{lemma}
\begin{proof}
    We proceed by induction on $n$. At $n=0$, \verb|chainTup| is $(m,-m)$, which is easily seen to satisfy \verb|StrongSSC| because $m\ne0$.

    Now suppose the proposition is true for $n=k$. \verb|chainTup| at $n=k+1$ is then
    \begin{equation*}
        \left(p_1,\dots,p_k,p_{k+1},m,-\left(m+\sum_{i=1}^{k+1} p_i\right)\right).
    \end{equation*}
    
    The third-last and last elements of this tuple form a subsum block by \verb|pair_of_sum_natAbs_lt|, because another induction argument on the given inequalities shows that
    \begin{equation*}
        m+\sum_{i=1}^k p_i<p_{k+1}.
    \end{equation*}

    Summing the subsum block into one element and permuting gives \verb|chainTup| at $n=k$, which satisfies \verb|StrongSSC| by the inductive hypothesis. The inductive step follows by \verb|strongSSC_tupReduce|.
\end{proof}

The prime chain used in the odd case of \cite{strongN} (cf. section \ref{oddcase}) conservatively stipulates $3p_{i-1}<p_i$ rather than $2p_{i-1}<p_i$.

\subsection{VW pairs} \label{vwpair}

Recall that we want to produce tuples in \verb|factorFreeTuples F n|. By setting the seed $s$ larger than all elements of $F$, the prime chain in \verb|chainTup| will satisfy the forbidden factor condition, but the same cannot be said of the two appended entries, which may also share common factors with each other or the prime chain. Lemma 2.2 in \cite{strongN} tackles this issue by splitting the last entry $-(m+\sum_{i=1}^n p_i)$ into two numbers $v$ and $-w$ satisfying certain properties. As formalised in Redhill it reads as follows:
\begin{lemma}[\href{https://github.com/Parcly-Taxel/Redhill/blob/arxiv-fixed/Redhill/Common/VWPair.lean\#L167}{\Verb|VWPair.of|}, \href{https://github.com/Parcly-Taxel/Redhill/blob/arxiv-fixed/Redhill/Common/VWPair.lean\#L176}{\Verb|VWPair.of\_coprime|}]
    Given two natural numbers $u$ and $m$, there exist two other natural numbers $v$ and $w$ such that $w=u+v$, $m<v$, $w$ is odd and $k$ does not divide $v$ or $w$ if $3\le k\le m$. If in addition $0<u\le m$, $v$ and $w$ can be chosen to be coprime.
\end{lemma}
\begin{proof}
    We choose $v=m+1+s$ and $w=m+1+u+s$, where $s$ is a natural number ``shift'' to be determined. They clearly satisfy $w=u+v$ and $m<v$, so we need only consider the non-divisibility conditions for $s$.

    Given any integer $p\ge3$, the pigeonhole principle applied to residues modulo $p$ provides a number $s_p$ with
    \begin{align*}
        m+1+s_p&\not\equiv0\bmod p\\
        m+1+u+s_p&\not\equiv0\bmod p.
    \end{align*}

    When $p=4$, simple casework provides a number $s_4$ with
    \begin{align*}
        m+1+s_4&\not\equiv0\bmod 4\\
        m+1+u+s_4&\not\equiv0\bmod\mathbf2.
    \end{align*}

    The Chinese remainder theorem (\href{https://github.com/leanprover-community/mathlib4/blob/9e6b3aac99b624d10c84653ab9c5357283b9b3b8/Mathlib/Data/Nat/ChineseRemainder.lean#L165}{\Verb|Nat.chineseRemainderOfFinset|}) then provides a number $s$ satisfying $s\equiv s_4\bmod 4$ and $s\equiv s_p\bmod p$ for all primes $3\le p\le m$. Since every number in $[3,m]$ is divisible by one of these moduli, using this $s$ satisfies the non-divisibility conditions.

    To show $v$ and $w$ are coprime when $0<u\le m$, suppose $p$ is a common prime factor. $p\ne2$ since $w$ is odd, so $p>m$ because neither $v$ nor $w$ is divisible by any number in $[3,m]$. But $p\mid u$ since $w=u+v$, so $p\le u$, contradicting $u\le m$.
\end{proof}

We record $v$, $w$ and their properties in \href{https://github.com/Parcly-Taxel/Redhill/blob/arxiv-fixed/Redhill/Common/VWPair.lean#L35-L49}{a structure} with parameters $u$ and $m$ as a reminder that the constructions in section \ref{disproof} only need the properties, not the exact values:
\begin{lstlisting}
/-- `VWPair u m` holds `v` and `w` and states that `-u, m, v, -w` satisfy the conditions
in Lemma 2.2. -/
structure VWPair (u m : ℕ) where
  /-- `v` in the paper -/
  v : ℕ
  /-- `w` in the paper, **negated** -/
  w : ℕ
  /-- `u = v + w` in the paper -/
  eq_add : w = u + v
  /-- `m < v` -/
  m_lt_v : m < v
  /-- `w` is odd -/
  w_odd : Odd w
  /-- No number in `[3,m]` divides `v` or `w` -/
  not_dvd (k) (hk : k ∈ Icc 3 m) : ¬k ∣ v ∧ ¬k ∣ w
\end{lstlisting}

These properties allow us to prove an extension of \verb|strongSSC_chainTup|.
\begin{lemma}[\href{https://github.com/Parcly-Taxel/Redhill/blob/arxiv-fixed/Redhill/Common/VWPair.lean\#L258}{\Verb|strongSSC\_vwTup|}]
    Let $m$ and $s$ be natural numbers with $0<m\le s$ and $p_1,\dots,p_n$ be primes with $s<p_1$ and $2p_{i-1}<p_i$ for $i=2,\dots,n$. Let $B$ be another natural number such that $m+\sum_{i=1}^n p_i\le B$ and let $v$ and $w$ be the corresponding fields in a \mmlean{VWPair $(m+\sum_{i=1}^n p_i)$ $B$}.

    Then \verb|vwTup| $(p_1,\dots,p_n,v,-w,m)$ satisfies \verb|StrongSSC|.
\end{lemma}
\begin{proof}
    $v$ and $-w$ form a subsum block by \verb|pair_of_sum_natAbs_lt|. It is clear that $m+\sum_{i=1}^n p_i>0$; expanding the properties of $v$ and $w$ gives
    \begin{align*}
        \mathtt{eq\_add}&:\ w=m+\sum_{i=1}^n p_i+v\\
        \mathtt{m\_lt\_v}&:\ B<v
    \end{align*}
    so $0<m+\sum_{i=1}^n p_i<v<w$, the strict inequality conditions in $\verb|pair_of_sum_natAbs_lt|$ are met and $v\cdot-w<0$. Summing the subsum block into one element and permuting gives \verb|chainTup|, so \verb|vwTup| satisfies \verb|StrongSSC| by \verb|strongSSC_tupReduce| and \verb|strongSSC_chainTup|.
\end{proof}

The statement of Lemma 2.2 in \cite{strongN} unnecessarily negates $u$ and $w$; it is more parsimonious to let $u$ be a natural number and use $-u$ than to let $u$ be a non-positive integer and use $u$. It also has two unused conditions, $2\le m$ and $w\le(m+1)q$ in Redhill's notation where $q$ is the product of all primes at most $p$, and a tighter lower bound than necessary of $q<v$. Its proof proceeds through an imperative algorithm that considers each relevant modulus in turn -- a less efficient approach compared to the Chinese remainder theorem.

\section{The disproof} \label{disproof}

From \verb|vwTup|, the final step to obtain the constructions in \cite{strongN} splits $m$ into a few univariate polynomials whose radicals are of significantly lesser degree, following a result of Konyagin \cite{brow00}. Each tuple in a constructed infinite sequence (lying in \verb|factorFreeTuples F n|) thus consists of three parts:
\begin{itemize}
    \item A \textbf{constant} (depending only on $n$ and $F$) prime chain (section \ref{primechain}) whose length varies with $n$
    \item $v$ and $-w$ from section \ref{vwpair}, which are also \textbf{constant}
    \item A \textbf{varying} polynomial head whose evaluation at different points generates the sequence
\end{itemize}

The evaluation points are chosen to further reduce the degree of a polynomial upper bound on $\operatorname{rad}(\prod_i a_i)$ in the sequence-generating variable to below the degree of $\max_i|a_i|$ in the same variable. If the two degrees are $d$ and $c$ respectively with $d<c$, applying equation \ref{eq:quality} shows that the tuple qualities tend to $c/d>1$, which is therefore a lower bound on the quality of \verb|factorFreeTuples F n|. The disproof of Ramaekers's conjecture for the chosen $n$ follows since \verb|factorFreeTuples F n| is a subset of \verb|ramaekersTuples n|.

\subsection{Odd case} \label{oddcase}

\href{https://github.com/Parcly-Taxel/Redhill/blob/arxiv-fixed/Redhill/Odd/Defs.lean#L42-L50}{This construction} only works for tuples of length $n+5$ with $n$ \textbf{even}\footnote{In this subsection and the next $n$ is the length of the \textit{prime chain} $p_1,\dots,p_n$, not the tuple's total length.} and $5,10\not\in F$ (recall that all forbidden factors, which are collected in $F$, are at least $3$). It uses the identity
\begin{equation*}
    8=(x-1)^5+10(x^2+1)^2-(x+1)^5,
\end{equation*}
so we call \verb|VWPair.of| with $u=8+\sum_{i=1}^n p_i$ and $m=\max(\{u\}\cup F)$. The prime chain's first member $p_1$ is greater than $\max(\{8\}\cup F)$.
\begin{lstlisting}
/-- The sum of `tup` over all indices save `n` and `n + 1`,
i.e. the input `u` to `VWPair`. -/
def U : ℕ := 8 + ∑ i ∈ range n, primeChain (max 8 (F.sup id)) i

/-- The `VWPair` generated from the inputs `u = U n F, m = max (U n F) (F.sup id)`. -/
def VW : VWPair (U n F) (max (U n F) (F.sup id)) := .of ..

/-- The sequence of `(n + 5)`-tuples containing an infinite subsequence
in `factorFreeTuples` whose qualities tend to `5 / 3`,
assuming `n` is even and `0, 1, 2, 5, 10 ∉ F`. -/
def tup (x : ℤ) (i : Fin (n + 5)) : ℤ :=
  i.addCases (primeChain (max 8 (F.sup id)) ·.1) fun
    | 0 => (VW n F).v
    | 1 => -(VW n F).w
    | 2 => (x - 1) ^ 5
    | 3 => 10 * (x ^ 2 + 1) ^ 2
    | 4 => -(x + 1) ^ 5
\end{lstlisting}

With the properties of prime chains and \verb|VWPair|, we can easily show that \verb|tup|'s entries sum to $0$, as well as \textit{a priori} lower bounds
\begin{equation*}
    11\le p_i\qquad8\le u\qquad 9\le v\qquad 17\le w.
\end{equation*}

We now require $x$ to be a multiple of
\begin{equation*}
    y=10\Biggr(\prod_{f\in F}f\Biggr)\left(\prod_{i=1}^n p_i\right)\cdot v\cdot w\ge1530.
\end{equation*}
\begin{lstlisting}
/-- We require `x` in `tup` to be a multiple of this number,
an optimised version of the paper's `y`. -/
def Y : ℕ :=
  10 * (F.erase 0).prod id *
  (∏ i ∈ range n, primeChain (max 8 (F.sup id)) i) * (VW n F).v * (VW n F).w
\end{lstlisting}

With this constraint, \verb|tup| has pairwise coprime entries and satisfies the forbidden factor condition:
\begin{itemize}
    \item Since $10$, all $p_i$, $v$ and $w$ divide $x$, $(x-1)^5$ and $-(x+1)^5$ are coprime with every other element of \verb|tup| and with each other.
    \item Since $11\le p_i$ and all $p_i$ divide $x$, $10(x^2+1)^2$ is coprime to all the $p_i$. Since $5\nmid v,w$ but $v,w\mid x$, $10(x^2+1)^2$ can have a common factor of at most $2$ with $v$ and $w$. But $w$ is odd by construction and \textbf{$n$ being even makes $v$ odd}\footnote{$u$ is the sum of $8$ and an even number of odd primes, so is even. $w=u+v$, so $v$ must be odd.}, so $10(x^2+1)^2$ is coprime with every other element of \verb|tup|. This is the only place in the odd case where the evenness of $n$ is used.
    \item \verb|VWPair.of_coprime| shows that $v$ and $-w$ are coprime. They are not divisible by any number in $[3,m]$ and $11\le p_i<u\le m$, so they are also coprime with all $p_i$. The $p_i$ are obviously pairwise coprime among themselves, completing the proof that \verb|tup|'s entries are pairwise coprime.
    \item For any factor $f\in F$ we have $f<p_i$, $3\le f\le m$ and $f\mid x$. These show that all $p_i$, the pair $(v,-w)$ and the polynomial head respectively avoid factors in $F$. We need $5,10\not\in F$ because of the entry $10(x^2+1)^2$.
\end{itemize}

We prove one more lemma before selecting evaluation points for the polynomial head.
\begin{lemma}[\href{https://github.com/Parcly-Taxel/Redhill/blob/arxiv-fixed/Redhill/Odd/Subsum.lean\#L215}{\Verb|OddCase.strongSSC\_tup|}]
    If $|x|\ge y$, \verb|tup| satisfies \verb|StrongSSC|.
\end{lemma}
\begin{proof}
    Knowing \textit{a priori} that $1530\le y$, we prove these estimates for the polynomial head's entries:
    \begin{align*}
        21|x|^4\le\,&|(x-1)^5|\\
        &|10(x^2+1)^2|\le20|x|^4\\
        21|x|^4\le\,&|-(x+1)^5|
    \end{align*}

    Then we show that for $b_1,b_2,b_3\in\{-1,0,1\}$ not all equal,
    \begin{equation*}
        y\le|b_1\cdot(x-1)^5+b_2\cdot10(x^2+1)^2+b_3\cdot-(x+1)^5|.
    \end{equation*}

    The sum of magnitudes of \verb|tup|'s remaining entries, $\sum_{i=1}^n p_i+v+w$, is strictly less than $y$, so the polynomial head's entries form a subsum block by \verb|of_sum_natAbs_lt|. Summing this block into one element gives an instance of \verb|vwTup| satisfying \verb|StrongSSC|; the result follows by \verb|strongSSC_tupReduce|.
\end{proof}

We now know that if $x$ is a non-zero multiple of $y$, \verb|tup| lies in \verb|factorFreeTuples F (n + 5)|. To obtain suitable evaluation points we consider the Pell-type equation $(ys)^2+1=(y^2+1)t^2$, which has \href{https://github.com/Parcly-Taxel/Redhill/blob/arxiv-fixed/Redhill/Odd/Pell.lean#L20}{infinitely many solutions} given by
\begin{align*}
    (s_0,t_0)&=(1,1)\\
    (s_{k+1},t_{k+1})&=((2y^2+1)s_k+(2y^2+2)t_k,2y^2s_k+(2y^2+1)t_k),
\end{align*}
and generate the sequence of tuples as $x_k=s_ky$. Note that $0<t_k\le s_k\le x_k$.

The radical of the product of \verb|tup|'s entries using $x_k$ is bounded as
\begin{align*}
    \operatorname{rad}\Biggr(\Biggr(\prod_{i=1}^n p_i\Biggr)\cdot v\cdot w\cdot(x_k-1)^5\cdot10(x_k^2+1)^2\cdot(x_k+1)^5\Biggr)&\le C\cdot(x_k-1)\cdot\operatorname{rad}(x_k^2+1)\cdot(x_k+1)\\
    &=C\cdot(x_k-1)\cdot\operatorname{rad}((y^2+1)t_k^2)\cdot(x_k+1)\\
    &\le C'\cdot(x_k-1)\cdot t_k\cdot(x_k+1)\\
    &\le C'\cdot(x_k-1)\cdot x_k\cdot(x_k+1)
\end{align*}
where $C$ and $C'$ are constants, while the maximum absolute value is $(x_k+1)^5$. Applying the argument at the beginning of section \ref{disproof} leads to Theorem 1.13 in \cite{strongN}:
\begin{theorem}[\href{https://github.com/Parcly-Taxel/Redhill/blob/arxiv-fixed/Redhill/Odd/Main.lean\#L155}{\Verb|quality\_factorFreeTuples\_ge\_of\_odd\_of\_disjoint|}]
    For all odd $n\ge5$ and all finsets $F$ of natural numbers not containing $0$, $1$, $2$, $5$ or $10$, \verb|factorFreeTuples F n| has quality at least $5/3$. Ramaekers's conjecture is thus false for these values of $n$.
\end{theorem}

At the time of writing Redhill does not have a formalisation of Theorem 2.3 in \cite{strongN}, which compared to Theorem 1.13 allows $5$ and $10$ in $F$ at the cost of a weaker lower bound on the quality. It is not needed to disprove Ramaekers's conjecture.

The proof of Theorem 1.13 in \cite{strongN} uses $3(8+\max(F))<p_1$ instead of $\max(\{8\}\cup F)<p_1$ (where $\max(F)=0$ if $F$ is empty) and $y=w!$ instead of the product of necessary factors that we found easier to understand and manipulate. Its arguments for the constructed tuples satisfying \verb|StrongSSC| are serpentine and highly non-rigorous -- subsum blocks were introduced in part to tidy these arguments for formalisation in Lean.

\subsection{General case} \label{generalcase}

\href{https://github.com/Parcly-Taxel/Redhill/blob/arxiv-fixed/Redhill/General/Defs.lean#L57-L66}{This construction} works for tuples of length $n+6$ and any $F$ whose elements are all at least $3$. Define
\begin{equation*}
    y=33330\prod_{f\in F}f\qquad x=(y+1)^{h!}
\end{equation*}
where varying $h$ generates the sequence of tuples. We use the identity
\begin{equation*}
    (100y-2)y^5 = (x^2+10y^3)^2 + (10y-1)x^4 + (x-y)^5 - (x+y)^5
\end{equation*}
and call \verb|VWPair.of| with $u=m=(100y-2)y^5+\sum_{i=1}^n p_i$. The prime chain's first member $p_1$ is greater than $100y^6$.
\begin{lstlisting}
/-- An optimised version of the paper's `y`. -/
def Y : ℕ := 33330 * (F.erase 0).prod id

/-- `x` in the paper, but using the optimised `y`. -/
def X : ℕ := (Y F + 1) ^ h !

/-- The sum of `tup` over all indices save `n` and `n + 1`,
i.e. the input `u` to `VWPair`. -/
def U : ℕ := (100 * Y F - 2) * Y F ^ 5 + ∑ i ∈ range n, primeChain (100 * Y F ^ 6) i

/-- The `VWPair` generated from the inputs `u = m = U n F`. -/
def VW : VWPair (U n F) (U n F) := .of ..

/-- The sequence of `(n + 6)`-tuples whose tail is in `factorFreeTuples`
and has quality tending to `5 / 4`. -/
def tup (i : Fin (n + 6)) : ℤ :=
  i.addCases (primeChain (100 * Y F ^ 6) ·.1) fun
    | 0 => (VW n F).v
    | 1 => -(VW n F).w
    | 2 => (X F h ^ 2 + 10 * Y F ^ 3) ^ 2
    | 3 => (10 * Y F - 1) * X F h ^ 4
    | 4 => (X F h - Y F) ^ 5
    | 5 => -(X F h + Y F) ^ 5
\end{lstlisting}

\newpage

Like the odd case, we deduce from the properties of prime chains and \verb|VWPair| that this \verb|tup| (\verb|GeneralCase.tup| in full, to disambiguate from \verb|OddCase.tup|) has entries summing to $0$, as well as the \textit{a priori} bound $33330\le y$. Unlike the odd case where all tuples in the construction are fully defined, we only show that \verb|tup| is in \verb|factorFreeTuples F (n + 6)| if $h$ is sufficiently large. The canonical way to formalise this last notion is \mmlean{∀ᶠ h in atTop} \cite{mil-filters}, and the \href{https://github.com/leanprover-community/mathlib4/blob/95142c9c4fdb4ff9632553901cfc82d5c84a8ad4/Mathlib/Order/Filter/Defs.lean#L544-L545}{\Verb|filter\_upwards|} tactic is useful in manipulating hypotheses using the notion.

We first prove that \verb|tup| has pairwise coprime entries for sufficiently large $h$. Note that $y\pm1$ is coprime to $10y\pm1$ regardless of sign choices because $3,11\mid y$. By Euler's theorem there exists a least positive exponent $k$ for each sign choice such that $(y\pm1)^k\equiv1\bmod(10y\pm1)$, so if $h$ is greater than all these exponents -- \textit{sufficiently large} -- $(y\pm1)^{h!}\equiv1\bmod(10y\pm1)$ for all sign choices. In particular, $x\equiv1\bmod(10y\pm1)$.

The same arguments as in the odd case show that all pairs between entries outside the polynomial head (the $p_i$, $v$ and $-w$) are coprime. The remaining pairs are handled as follows:
\begin{itemize}
    \item \textit{An entry in the polynomial head and an entry outside it.} Because $y$ is even, $x$ and all entries of \verb|tup| except possibly $v$ are odd, so we need only rule out common \textit{odd prime} factors. By construction, all such factors of entries outside the polynomial head are greater than $(100y-2)y^5$.

    Let $p$ be one such factor. Since $p>(100y-2)y^5>y+1$, Fermat's little theorem implies that $x=(y+1)^{h!}\equiv1\bmod p$ for sufficiently large $h$, so
    \begin{align*}
        x^2+10y^3\equiv1+10y^3&\not\equiv 0\bmod p\\
        10y-1&\not\equiv 0\bmod p\\
        x-y\equiv1-y&\not\equiv0\bmod p\\
        x+y\equiv1+y&\not\equiv0\bmod p
    \end{align*}
    where the non-congruences arise because $p$ is strictly greater than all residues. Therefore, $p$ cannot be a factor of any entry in the polynomial head.
    \item \textit{$(x-y)^5$ and $-(x+y)^5$}. This reduces to showing $\gcd(x-y,x+y)=\gcd(2y,x+y)=1$. Since $x+y$ is odd this reduces further to $\gcd(x,y)=1$, which is true by unfolding $x$'s definition. This is the only pair out of the pairs within the polynomial head that is coprime for all $h$, not just sufficiently large $h$.
    \item \textit{$(10y-1)x^4$ and $(x-y)^5$/$-(x+y)^5$.} We have just shown that $\gcd(x,y)=1$, so $\gcd(x,x\pm y)=1$. For sufficiently large $h$ we know that $x\equiv1\bmod(10y-1)$, so $x\pm y\equiv1\pm y\bmod(10y-1)$. We also know that $10y-1$ is coprime to $y\pm1$, so $10y-1$ is coprime to $x\pm y$ too.
    \item \textit{$(x^2+10y^3)^2$ and $(x-y)^5$/$-(x+y)^5$.} It suffices to show that $x^2+10y^3$ is coprime to $(x+y)(x-y)=x^2-y^2$, which by subtraction reduces to $\gcd((10y+1)y^2,x^2-y^2)=1$. The $y^2$ factor is handled as before using $\gcd(y,x\pm y)=\gcd(x,y)=1$. A very similar argument to the previous bullet point, using $x\equiv1\bmod(10y+1)$ for sufficiently large $h$, handles the $10y+1$ factor.
    \item \textit{$(x^2+10y^3)^2$ and $(10y-1)x^4$.} We need to show that $\gcd(x^2+10y^3,x)=\gcd(x^2+10y^3,10y-1)=1$. Reducing the first $\gcd$ gives $\gcd(x^2+10y^3,x)=\gcd(10y^3,x)$, which is $1$ because $\gcd(x,y)=1$ and $\gcd(x,10)=1$ from $10\mid y$. To show that the second $\gcd$ is $1$, we square the congruence $x\equiv1\bmod(10y-1)$ that holds for sufficiently large $h$ and obtain an integer $M$ with $x^2-1=(10y-1)M$. Then
    \begin{align*}
        \gcd(x^2+10y^3,10y-1)&=\gcd((10y-1)M+1+10y^3,10y-1)=\gcd(10y^3+1,10y-1)\\
        &=\gcd(y^2+1+(10y-1)y^2,10y-1)\\
        &=\gcd(y^2+1+(10y-1)(10y+1),10y-1)\\
        &=\gcd(101y^2,10y-1).
    \end{align*}

    It is almost trivial to show that $\gcd(y,10y-1)=1$, while $\gcd(101,10y-1)=1$ because $101\mid y$. This completes the proof of \verb|tup| having pairwise coprime entries for sufficiently large $h$.
\end{itemize}

Proving the forbidden factors condition is simple. For any $f\in F$, $f\mid y$ and hence $x\equiv1\bmod f$, which in turn leads to all entries in the polynomial head being $\pm1\not\equiv0\bmod f$. Since
\begin{align*}
    3\le f<y&<(100y-2)y^5\le m<v<w\\
    y&<100y^6<p_1
\end{align*}
and $v$ and $w$ are not divisible by any number in $[3,m]$, all entries outside the polynomial head are not divisible by $f$ either.

Finally, we prove that \verb|tup| satisfies \verb|StrongSSC|. The entries outside the polynomial head are constant, so $x$ exceeds the sum of their magnitudes for sufficiently large $h$. The remainder of this sub-proof is almost identical to what we did in the odd case: show that the polynomial head's entries form a subsum block by \verb|of_sum_natAbs_lt|, sum that block to obtain an instance of \verb|vwTup| satisfying \verb|StrongSSC| and conclude with \verb|strongSSC_tupReduce|. Here we use these estimates:
\begin{align*}
    &|(x^2+10y^3)^2|\le4x^4\\
    12yx^4\le&|(x-y)^5|\\
    12yx^4\le&|-(x+y)^5|
\end{align*}

We are now ready to prove Theorem 1.14 in \cite{strongN}.
\begin{theorem}[\href{https://github.com/Parcly-Taxel/Redhill/blob/arxiv-fixed/Redhill/General/Main.lean\#L110}{\Verb|quality\_factorFreeTuples\_ge|}]
    For all $n\ge6$ and all finsets $F$ of natural numbers not containing $0$, $1$ or $2$, \verb|factorFreeTuples F n| has quality at least $5/4$. Ramaekers's conjecture is thus false for these values of $n$.
\end{theorem}
\begin{proof}
    Recall that entries outside the polynomial head are constant, so they only contribute a constant factor $C$ to the upper bound on the radical of the product of \verb|tup|'s entries:
    \begin{align*}
        \operatorname{rad}(C\cdot(x^2+10y^3)^2\cdot(10y-1)x^4\cdot(x-y)^5\cdot(x+y)^5)&\le C'\cdot(x^2+10y^3)\cdot\operatorname{rad}(x)\cdot(x-y)\cdot(x+y)\\
        &=C'\cdot(x^2+10y^3)\cdot\operatorname{rad}((y+1)^{h!})\cdot(x^2-y^2)\\
        &=C'\cdot(x^2+10y^3)\cdot\operatorname{rad}(y+1)\cdot(x^2-y^2)\\
        &=C''\cdot(x^2+10y^3)\cdot(x^2-y^2)
\end{align*}
where $C'$ and $C''$ are constants. This is an upper bound of degree $4$, while \verb|tup|'s maximum absolute value is $(x+y)^5$ of degree $5$. Applying the argument at the beginning of section \ref{disproof} gives the result.
\end{proof}

The proof of Theorem 1.14 in \cite{strongN} confuses itself by negating the entire tuple (``Let $a_7,a_8,\dots,a_n$ be negative odd prime numbers\textellipsis'') and introducing two integer variables $\ell\ge11$ and $t>101$, then setting $y=\ell!\cdot t$ and assuming $F=\{3,4,\dots,\ell\}$ ``[as] enlarging $F$ only makes the statement harder to prove''. There is no need to enlarge $F$, since $x$'s construction allows proving the forbidden factors condition as long as $\prod_{f\in F} f\mid y$. The choice of $y$ also leads to an error in the proof of $(x^2+10y^3)^2$ and $(10y-1)x^4$ being coprime: ``As $101$ is prime and $10y-1>101$, they have no common factor.'' (Take $\ell=11$ and $t=190$.) Including an explicit factor of $33330=3\cdot10\cdot11\cdot101$ in $y$ corrects this error.

\subsection{Miscellaneous results}

Besides the two major theorems from \cite{strongN} disproving Ramaekers's conjecture for all $n\ge5$, three other results are formalised in Redhill:
\begin{itemize}
    \item \href{https://github.com/Parcly-Taxel/Redhill/blob/arxiv-fixed/Redhill/KonyaginPrelude.lean}{Theorem 2.1} in \cite{strongN}, which proves a lower bound of $3/2$ on the quality of \verb|ramaekersTuples 5|. It is strictly weaker than Theorem 1.13, but was meant as ``a warm-up and illustration of Konyagin's technique'' before proceeding to the stronger theorem. Accordingly, Theorem 2.1 was formalised first to make sure that Redhill's definitions were canonical.
    \item The quality of \verb|nConjectureTuples n| is \textit{\href{https://github.com/Parcly-Taxel/Redhill/blob/arxiv-fixed/Redhill/BB94.lean}{at least}} $2n-5$. This was proved in the same paper \cite{bb94} that put forth the $n$-conjecture, which asserts that the quality is \textit{exactly} $2n-5$.
    \item Adamczewski's LLM-assisted disproof \cite{adamczewski} of Ramaekers's conjecture for $n=4$. We examine the original code in section \ref{llmcomp}.
\end{itemize}

\section{Upstreaming to \texttt{mathlib}} \label{upstreaming}

\verb|mathlib| is vast enough of a library that a paper has been written about growing and maintaining it sustainably \cite{growing-mathlib}, but its \href{https://leanprover-community.github.io/contribute/index.html#what-to-contribute-to-mathlib}{vague aspiration} of covering university-level mathematics, combined with the continuous development of Lean's kernel, means it will always have missing lemmas and definitions. Any sufficiently large project will reprove theory useful beyond its own scope, such as the real interpolation theorem in \cite{carleson}. In anticipation of upstreaming it to \verb|mathlib|, said theory should be put in a folder (typically called \verb|ToMathlib|, as Redhill and \cite{carleson} do) not importing any project-specific files.

The following pull requests (PRs) to \verb|mathlib| upstreaming material from Redhill have been merged:
\begin{itemize}
    \item \href{https://github.com/leanprover-community/mathlib4/pull/35673}{\#35673} added lemmas about the \href{https://github.com/leanprover-community/mathlib4/blob/a02a63c2100af866343c74550041885f8f39f099/Mathlib/RingTheory/Radical/Basic.lean#L144}{\Verb|radical|} of integers where there were previously none. This function appears in \verb|tupleQuality|, where we could have taken the absolute value first using \href{https://github.com/leanprover/lean4/blob/c29b6dda4f7c20e3eeaa717c4e565663c5cfa364/src/Init/Data/Int/Basic.lean#L329}{\Verb|Int.natAbs|} so as to only consider \verb|radical|s of natural numbers -- lemmas existed about the latter prior to this PR. For simplicity, we do not take the absolute value.
    \item \href{https://github.com/leanprover-community/mathlib4/pull/36340}{\#36340} added iff lemmas on when the product of two integers is positive, negative or non-positive, complementing \href{https://github.com/leanprover-community/mathlib4/blob/aa4180fd65f90dd50307a354941dfbee635cf9d2/Mathlib/Data/Int/Order/Basic.lean#L59-L60}{an existing lemma} about a non-negative product. The non-positive version is used to prove \verb|pair_of_sum_natAbs_lt|.
    \item \href{https://github.com/leanprover-community/mathlib4/pull/36799}{\#36799} added miscellaneous lemmas about primes and \verb|radical|.
    \item \href{https://github.com/leanprover-community/mathlib4/pull/36948}{\#36948} added convenience lemmas for \href{https://github.com/leanprover-community/mathlib4/blob/aa4180fd65f90dd50307a354941dfbee635cf9d2/Mathlib/RingTheory/Coprime/Basic.lean#L43}{\Verb|IsCoprime|}, specifically its preservation under subtracting multiples of one number from another, that $x$ is coprime to $x\pm1$ and that $x+1$ is coprime to $x-1$ if $2\mid x$. These lemmas facilitate the proofs of pairwise coprime entries in both major theorems of \cite{strongN}. At the request of \verb|mathlib| maintainers, corresponding lemmas for \href{https://github.com/leanprover-community/mathlib4/blob/aa4180fd65f90dd50307a354941dfbee635cf9d2/Mathlib/Algebra/Divisibility/Units.lean#L150}{\Verb|IsRelPrime|} were also added in the PR.
\end{itemize}

While upstreaming is good in general, it has certain limitations. Every PR to \verb|mathlib| must be reviewed by at least two maintainers before being merged, and those maintainers usually do not review PRs outside their areas of interest. The autoformalisation gold rush (see section \ref{introduction}) has created a surge of PRs to \verb|mathlib| from prompters of LLM-assisted projects, some of whom are unaware of \href{https://leanprover-community.github.io/contribute/values.html}{\Verb|mathlib|'s values}. Maintainers need time to separate high-quality PRs and provide feedback to authors of low-quality ones, so it may take a long time for an upstreaming PR to be merged.

\section{Comparison with LLM-assisted projects} \label{llmcomp}

Redhill and this paper are \textbf{completely} human-written. The library was created in February 2026, after Aristotle \cite{aristotle} and other autoformalisers had been made publicly available, but before those LLMs had seen extensive use. Formal proofs of both major theorems in \cite{strongN} were \href{https://github.com/Parcly-Taxel/Redhill/commit/a2ecf98bdc7cc64d5230af938e11e077aeb85199}{completed} in April that same year, but the library has been updated since then to stay abreast of developments in Lean.

Ramaekers's conjecture for $n=4$ was left open by \cite{strongN}. On 6 September 2026, an entirely LLM-written (primarily by GPT-6 Astra) Lean repository was published claiming to disprove it \cite{adamczewski} with a construction adapted from \cite{pakovich-zvonkin}. We have successfully integrated this claimed disproof into Redhill without any further LLM assistance, thereby conclusively validating it:
\begin{theorem}[\href{https://github.com/Parcly-Taxel/Redhill/blob/arxiv-fixed/Redhill/Four.lean\#L331}{\Verb|quality\_ramaekersTuples\_four\_ge|}]
    \verb|ramaekersTuples 4| has quality at least $9/8$, so Ramaekers's conjecture is false for $n=4$.
\end{theorem}

\begin{proof}
\href{https://github.com/Parcly-Taxel/Redhill/blob/arxiv-fixed/Redhill/Four.lean#L31-L37}{The construction} is broadly similar to those that we encountered in section \ref{disproof}, in that a tuple of polynomials is evaluated at specific points:
\begin{gather*}
    (a_0:=u^9, a_1:=(8-u)^5(u^2+20u+280)^2, a_2:=-105(2u-3)^6,\\
    a_3:=130032u^4 + 10728480u^3 - 202978980u^2 + 1238324220u - 2568934655)
\end{gather*}

It is an algebraic identity that $a_0+a_1+a_2+a_3=0$.
\begin{lstlisting}
/-- An abbreviation for the irreducible quartic polynomial that is the tuple's last term. -/
abbrev Q (u : ℤ) : ℤ :=
  130032 * u ^ 4 + 10728480 * u ^ 3 - 202978980 * u ^ 2 + 1238324220 * u - 2568934655

/-- The sequence of 4-tuples containing an infinite subsequence in `ramaekersTuples`
whose qualities tend to `9 / 8`. -/
def tup (u : ℤ) : Fin 4 → ℤ
  | 0 => u ^ 9
  | 1 => (8 - u) ^ 5 * (u ^ 2 + 20 * u + 280) ^ 2
  | 2 => -105 * (2 * u - 3) ^ 6
  | 3 => Q u
\end{lstlisting}

For each pair $X,Y$ of irreducible polynomials (including $-105$ in $a_2$) that are factors of different $a_i$, there exist two other polynomials $A,B$ and a constant $c$ such that $AX+BY=c$ with no common factor among all of $A,B,c$, so $X$ and $Y$ are coprime if they are coprime modulo $c$. Substituting $u=19$ makes all such pairs coprime modulo their corresponding $c$, and we can easily show that $u$ is coprime to $u^2+20u+280$ if $u\equiv19\bmod70$. Thus \verb|tup| has pairwise coprime entries if $u\equiv19\bmod M$, where
\begin{equation*}
    M=2\cdot3\cdot5\cdot7\cdot13\cdot23\cdot41\cdot103\cdot113\cdot127\cdot149
    \cdot197\cdot439\cdot1249\cdot395119\cdot649541\cdot399635598557777987
\end{equation*}
is the least common multiple of all moduli $c$.

When $u\ge19$, the inequalities $1<a_3<-a_2<a_0$ and $a_3<-a_1<a_0$ hold. It is then a simple matter to show that \verb|tup| satisfies the subsum condition (\textit{not} \verb|StrongSSC|) by considering each possible proper subset and calling the \verb|grind| tactic to prove that it cannot sum to $0$.

We have proved that \verb|tup| is in \verb|ramaekersTuples 4| if $u\ge19$ and $u\equiv19\bmod M$. An infinite sequence of evaluation points that always satisfies these conditions is $u_k=(35(2M+1)^k+3)/2$. To finish the proof of the lower bound on quality, we derive an upper bound of degree $8$ on $\operatorname{rad}(a_0a_1a_2a_3)$:
\begin{align*}
    \operatorname{rad}(a_0a_1a_2a_3)&\le 105\cdot u_k\cdot(8-u_k)\cdot(u_k^2+20u_k+280)\cdot\operatorname{rad}(2u_k-3)\cdot a_3\\
    &=105\cdot u_k\cdot(8-u_k)\cdot(u_k^2+20u_k+280)\cdot\operatorname{rad}(35(2M+1)^k)\cdot a_3\\
    &=C\cdot u_k\cdot(8-u_k)\cdot(u_k^2+20u_k+280)\cdot\underbrace{a_3}_{\text{degree 4}}
\end{align*}
where $C$ is a constant. The inequalities between the $a_i$ used to prove the subsum condition also show that the maximum absolute value is $a_0=u^9$. Once again, applying the argument at the beginning of section \ref{disproof} gives the result.
\end{proof}

We now compare the original code in \cite{adamczewski} with the same code as integrated into Redhill. We find some deficiencies in the former, which we turn into guidelines for good Lean formalisations:
\begin{itemize}
    \item \textit{Write canonical definitions.} They are meant to be unfolded downstream, so they should strive to minimise the inconvenience caused at that point.
    
    This is how the quality of a set of tuples is \href{https://github.com/tadamcz/n-conjecture-strong/blob/1995779c1885912ad77a1141c877556736ab60ef/StrongFour/Definitions.lean#L29-L30}{defined} in \cite{adamczewski}:
\begin{lstlisting}
noncomputable def qualityLimsup : EReal :=
  Filter.limsup (fun a : {a : Tuple // Admissible a} ↦ (quality a.1 : EReal))
    Filter.cofinite
\end{lstlisting}
    By abstracting the tuples in \verb|ramaekersTuples 4| into a \href{https://github.com/leanprover/lean4/blob/c29b6dda4f7c20e3eeaa717c4e565663c5cfa364/src/Init/Prelude.lean#L711}{\Verb|Subtype|} wrapping the membership predicate, rather than leaving them as a set defined by the predicate, a proof must deal with an extra layer of indirection when it unfolds \verb|qualityLimsup|. A similar comment applies to the use of \href{https://github.com/leanprover-community/mathlib4/blob/2da3ab0a20dc0c66f56a557e37df39f7d36cb953/Mathlib/Order/LiminfLimsup.lean#L71}{\Verb|Filter.limsup|} and \href{https://github.com/leanprover-community/mathlib4/blob/2da3ab0a20dc0c66f56a557e37df39f7d36cb953/Mathlib/Order/Filter/Cofinite.lean#L41}{\Verb|Filter.cofinite|} to express the limit superior, rather than \verb|sInf| on a ``dual'' set of quality bounds. The definition is also harder to generalise for using \verb|Subtype|: different sets of tuples of the same length have the same type (and so can be substituted in Redhill's \verb|quality|), whereas \verb|Subtype|s with different predicates are not equal.
    \item \textit{Inline short definitions.} Names are given to every irreducible non-constant factor appearing in the construction of \cite{adamczewski}:
\begin{lstlisting}
def a (u : ℤ) : ℤ := u
def b (u : ℤ) : ℤ := u - 8
def c (u : ℤ) : ℤ := u ^ 2 + 20 * u + 280
def ell (u : ℤ) : ℤ := 2 * u - 3
def d (u : ℤ) : ℤ := 130032 * u ^ 4 + 10728480 * u ^ 3 - 202978980 * u ^ 2 +
  1238324220 * u - 2568934655
\end{lstlisting}
    Except for \verb|d|, the brevity of the polynomials they represent means they do not achieve significant space savings over inlining, while their status as definitions requires an explicit unfolding step before certain tactics like \verb|ring| can be used on them. \verb|a| is particularly egregious in this regard, being merely the identity function over $\mathbb Z$.

    Long definitions like \verb|d| are fine, but if they will be unfolded frequently they should be declared \verb|abbrev|s to allow implicit unfolding, as Redhill's equivalent declaration \verb|Q| is.
    \item \textit{Use all tactics in moderation.} Like code in any other programming language, Lean code is read by humans. The most appropriate tactic does exactly what its informal rephrasing (see section \ref{tactic-vocab}) says it does, giving the strongest signal to a human about how the proof works.

    One violation of this guideline in \cite{adamczewski} is in \href{https://github.com/tadamcz/n-conjecture-strong/blob/1995779c1885912ad77a1141c877556736ab60ef/StrongFour/Result.lean#L95}{\Verb|qualityLimsup\_ge|}:
\begin{lstlisting}
theorem qualityLimsup_ge : ((9 / 8 : ℝ) : EReal) ≤ qualityLimsup := by
  ...
  have hq' : q < (9 / 8 : ℝ) := EReal.coe_lt_coe_iff.mp hq
  have hle : (q : EReal) ≤ qualityLimsup := by
    apply Filter.le_limsup_of_frequently_le'
    apply hcof.frequently
    exact (eventually_quality_gt q hq').frequently.mono fun n hn ↦
      EReal.coe_le_coe_iff.mpr hn.le
  exact not_le_of_gt hqL hle    
\end{lstlisting}
    All that the line beginning with \verb|have hq'| does is to rewrite the hypothesis \verb|hq|. It is better to call the rewriting tactic \verb|rw| on \verb|hq|:
\begin{lstlisting}
theorem qualityLimsup_ge : ((9 / 8 : ℝ) : EReal) ≤ qualityLimsup := by
  ...
  rw [EReal.coe_lt_coe_iff] at hq
  have hle : (q : EReal) ≤ qualityLimsup := by
    apply Filter.le_limsup_of_frequently_le'
    apply hcof.frequently
    exact (eventually_quality_gt q hq).frequently.mono fun n hn ↦
      EReal.coe_le_coe_iff.mpr hn.le
  exact not_le_of_gt hqL hle
\end{lstlisting}
\end{itemize}

\section{Conclusion} \label{conclusion}

The word ``formalisation'' with respect to a proof can be interpreted in two ways. One interpretation is ``illuminating the process'' -- paying careful attention to every construct and argument, distilling them into a poem that humans and machines can read equally well. We have followed this interpretation in Redhill, as have the contributors to Imperial's FLT project which is mentioned in section \ref{introduction}.

The other interpretation is ``obtaining code that verifies the final statement'', treating the existence of that code as the end of the matter. Many LLM-generated repositories following this interpretation \cite{platonic-classification,pegden-spgt,anthropic-flt,adamczewski} float around the Internet, unsure of where or how they fit into the mathematical literature.

We believe that mathematics will be better off when people adopt the first interpretation -- when they pick up Lean and make it their own.

\printbibliography

\end{document}